\documentclass{amsart}
\usepackage{amsmath}
\usepackage{amssymb}
\usepackage{amsthm}
\usepackage{verbatim}
\usepackage{graphicx}
\usepackage{graphics}
\usepackage{color}
\usepackage{enumerate}
\usepackage{mathrsfs}
\usepackage{booktabs}
\usepackage[toc,page]{appendix}
\usepackage{enumerate}
\usepackage{mathrsfs}
\usepackage{fancyhdr}
\usepackage{amsopn, amstext, amscd,pifont}
\usepackage{amssymb,bm, cite}
\usepackage[all]{xy}
\usepackage{color}
\usepackage{graphicx,hyperref}
\usepackage{float}
\usepackage{listings}
\usepackage{setspace}
\usepackage{url}
\usepackage{paralist}
\usepackage{subcaption}

\usepackage{graphicx,pgf}
\usepackage{float}

\usepackage{fancyhdr}

\theoremstyle{plain}

\newtheorem{theorem}{Theorem}[section]
\newtheorem{proposition}[theorem]{Proposition}
\newtheorem{lemma}[theorem]{Lemma}
\newtheorem{corollary}[theorem]{Corollary}

\theoremstyle{definition}

\newtheorem{example}[theorem]{Example}

\graphicspath{{images/}}

\newcommand{\bx}{\mathbf{x}}
\newcommand{\by}{\mathbf{y}}
\newcommand{\bc}{\mathbf{c}}
\newcommand{\bg}{\mathbf{g}}
\newcommand{\bv}{\mathbf{v}}
\newcommand{\bw}{\mathbf{w}}

\begin{document}
\title{Bounded hyperbolic components of bicritical rational maps}
\author{Hongming Nie and Kevin M. Pilgrim}
\address{Einstein Institute of Mathematics, The Hebrew University of Jerusalem}
\email{hongming.nie@mail.huji.ac.il}
\address{Department of Mathematics, Indiana University}
 \email{pilgrim@indiana.edu}
\date{\today}
\maketitle

\begin{abstract}
We prove that the hyperbolic components of bicritical rational maps having two distinct attracting cycles each of period at least two are bounded in the moduli space of bicritical rational maps. Our arguments rely on arithmetic methods.
\end{abstract}


\section{Introduction}
A degree $d\ge 2$ rational map $f: \mathbb{P}^1 \to \mathbb{P}^1$ defines a dynamical system on the complex projective line.  The moduli space $\mathrm{rat}_d:=\mathrm{Rat}_d/\mathrm{Aut}(\mathbb{P}^1)$ is the space of holomorphic conjugacy classes of degree $d$ rational maps, and is naturally an affine variety \cite[Theorem 4.36]{Silverman07}. A rational map is \textit{hyperbolic} if each critical point converges under iteration to an attracting cycle; this property is invariant under conjugacy. The set of hyperbolic maps is an open and conjecturally dense subset of $\mathrm{rat}_d$. A connected component $\mathcal{H}$ of this subset is called a \textit{hyperbolic component}.  We are interested in the general question of when a hyperbolic component $\mathcal{H}$ has compact closure in $\mathrm{rat}_d$. In particular, we are interested in the analogous question for natural dynamically defined subvarieties of $\mathrm{rat}_d$.

A rational map is \textit{bicritical} if it has exactly two critical points. Equivalently, $f=M\circ z^d \circ N$ for $M, N\in \mathrm{Aut}(\mathbb{P}^1)$.  Such maps were studied by Milnor \cite{Milnor00}.
In \cite[Corollary 1.3]{Milnor00}, he showed that the moduli space $\mathcal{M}_d\subset\mathrm{rat}_d$ of bicritical rational maps is biholomorphic to $\mathbb{C}^2$. Explicit coordinates are given by $(\sigma_1, \sigma_d)$ where the $\sigma_i$'s are the elementary symmetric functions of the multipliers of the $d+1$ fixed-points \cite[Remark 2.7]{Milnor00}. It follows that the inclusion $\mathcal{M}_d \hookrightarrow \mathrm{rat}_d$ is proper.  Hyperbolic components in $\mathcal{M}_d$ are classified as in the quadratic case \cite{Milnor93}.  A hyperbolic component $\mathcal{H}\subset\mathcal{M}_d$ is of \textit{type D} if each map in $\mathcal{H}$ has two distinct attracting cycles; for convenience, we say it  is of \textit{strict type D} if neither of the attracting cycles is  a fixed point. The properness of the embedding $\mathcal{M}_d \hookrightarrow \mathrm{rat}_d$ implies that $\mathcal{H}$ has compact closure in $\mathcal{M}_d$ if and only if its image has compact closure in $\mathrm{rat}_d$.

Our main result is the following, which extends Epstein's boundedness result for strict type $D$ quadratic rational maps \cite[Theorem 1]{Epstein00}.
\begin{theorem}\label{main}
Let $\mathcal{H}\subset\mathcal{M}_d$ be a strict type D hyperbolic component. Then $\mathcal{H}$ has compact closure in $\mathcal{M}_d$.
\end{theorem}

In rough outline, our arguments and Epstein's are similar. Suppose to the contrary that there exists a degenerating sequence $f_1, f_2, \ldots$ of rational maps in $\mathcal{H}$. Since the sequence lies in a single hyperbolic component $\mathcal{H}$, there is a period $n$ such that each $f_k$ has an attracting cycle of period $n$. In particular, the multiplier of an $n$ cycle remains bounded along this sequence. From this we extract a so-called rescaling limit, $g=\lim g_k$ where $g_k = M_k\circ f^q_k\circ M_k^{-1}$, $M_k \in \mathrm{Aut}(\mathbb{P}^1)$, and $\deg(g) \geq 2$.  We analyze the possibilities for the limits of the attracting cycles for the $g_k$ and how such limits relate to the dynamics of $g$. Using the existence of a second cycle of bounded multiplier and a case-by-case analysis, we derive an over-determined set of constraints on the dynamics of critical points of $g$, and hence a contradiction.

However, our arguments differ from Epstein's in key respects.

First, Epstein works with sequences $f_k$.    Our argument exploits right away the type $D$ hypothesis to conclude that, via the multipliers of the two attracting cycles, $\mathcal{H}$ is algebraically equivalent to $\mathbb{D} \times \mathbb{D}$ and is in particular semi-algebraic \cite{Milnor12, Milnor14}.  It follows that if $\mathcal{H}$ is unbounded, then we may take the family $f_{k}$ to be of the form $f_{t_k}$ where $f_t$ is a holomorphic one-parameter family, $t \in \mathbb{D}$, with $f_t$ degenerating as $t \to 0$.


Second, Epstein  derives the existence of rescaling limits and the relationship between the dynamics of $g$ and of the $g_k$ via detailed analytic estimates.  In contrast, we rely on much softer arithmetic methods.  The holomorphic family $f_t$ induces a dynamical system $\mathbf{f}: \mathbf{P}^1 \to \mathbf{P}^1$ on the Berkovich space associated to a non-Archimedean field that is a completion of the field of Puiseux series \cite{Kiwi15}. Very loosely, the map $\mathbf{f}$ captures the asymptotics as $t \to 0$ of the family $f_t$ at all possible algebraically parameterized locations and scales.  Kiwi \cite{Kiwi14} gives a thorough analysis of the global structure of these Berkovich dynamical systems in degree two. We begin by generalizing Kiwi's results from the quadratic to the bicritical case; see section \ref{Berh-dyn}.  From the existence of a multiplier that remains bounded, we conclude the existence of a rescaling limit $g$--equivalently, a type II repelling cycle for $\mathbf{f}$; see Lemma \ref{repelling-existence}.  Using a case-by-case analysis, we then derive an over-determined set of constraints on the critical dynamics of $g$.  This step uses  arithmetic results of Rivera-Letelier (see Theorem \ref{fp-number}) and Epstein's refined version of the Fatou-Shishikura Inequality (see Theorems \ref{FSI} and \ref{refined-FSI}).

Finally, we do not know if Epstein's original analytic arguments may be extended to the bicritical case. Arguments of Kabelka \cite{Kabelka10} using a convenient normal form with a distinguished fixed-point should show that boundedness of the multiplier of a cycle yields a nonlinear rescaling limit, which is the first step in both Epstein's and our argument. When $d>2$ there are more than two other fixed-points and estimating their multipliers is more difficult. It would be interesting to have a close interpretation of Epstein's original arguments in arithmetic terms.

We conclude this introduction with a brief survey of boundedness and unboundedness results for hyperbolic components. Makienko \cite[Theorem A]{Makienko00} gives sufficient general topological-dynamical conditions for unboundedness of hyperbolic components; see also \cite{Tan02}. Applied to bicritical maps, this immediately yields that components of type A (maps with an attracting fixed-point) and certain components of type B (maps with both critical points in the immediate basin of an attracting cycle of period at least two) are unbounded. However, we do not know if components of type C (so-called capture components, in which both critical points are attracted to an attracting periodic cycle, but only one lies in the immediate basin) can be unbounded, or if they can be bounded. Unlike components of type D, components of types A, B, and C are transcendental objects. If a component $\mathcal{H}$ of type C were unbounded, we do not know how to extract a degenerating holomorphic family $f_t$ such that $f_{t_k} \in \mathcal{H}$ for some sequence $t_k \to 0$. \\

\noindent {\bf Acknowledgement.} We thank Jan Kiwi for encouragement and helpful insights.

%

\section{Berkovich dynamics of bicritical rational maps}\label{Berh-dyn}

In this section, we first recall some concepts and known results from Berkovich dynamics, and then extend Kiwi's classification results from the quadratic to the bicritical case.

\subsection*{Berkovich space} See \cite{Baker10, Kiwi14} for details.
The field of Laurent series $\sum_{n \geq N}a_nt^n$ has an algebraic completion whose elements are Puiseux series.  It has a further algebraic and metric completion given by a field, here denoted $\mathbb{L}$, defined as follows. The elements of $\mathbb{L}$ are formal series $z=a_{q_0}t^{q_0}+a_{q_1}t^{q_1}+\ldots, q_n \in \mathbb{Q}, q_n \uparrow \infty$. The metric is given by the non-Archimedean absolute value $|z|:=\exp(-q_0)$.  With the metric topology induced by the absolute value, the projective space $\mathbb{P}^1_{\mathbb{L}}$ is totally disconnected and not locally compact. To remedy this, it is compactified by the Berkovich projective space over $\mathbb{L}$, denoted $\mathbf{P}^1$, which is a uniquely arcwise connected, compact, Hausdorff topological space. The points in $\mathbb{P}^1_{\mathbb{L}}$ form a dense subset of $\mathbf{P}^1$, which are called \emph{type I} points. We denote by $\mathbb{H}:=\mathbf{P}^1\setminus\mathbb{P}^1_{\mathbb{L}}$.  The branch points (those whose complement has three or more components) in $\mathbf{P}^1$ are called \emph{type II} points.  A type II point corresponds to a closed ball in $\mathbb{L}$ with radius in the value group of $\mathbb{L}$. The type II point in $\mathbf{P}^1$ corresponding to the closed unit disk in $\mathbb{L}$ is called the \textit{Gauss point} and denoted by $\mathbf{g}$.

Suppose $\xi \in \mathbf{P}^1$. Recall that $\mathbf{P}^1$ is a tree.  A connected component of $\mathbf{P}^1\setminus \{\xi\}$ is therefore called a \emph{direction} at $\xi$; such a component is also sometimes called an \emph{open Berkovich ball}.  Abusing notation and terminology, we denote such a direction by $\vec{v}$. While a direction is determined by the basepoint $\xi$ and any element in the corresponding ball, a direction is not an infinitesimal object.  The set of such directions is called the \emph{tangent space} to $\xi$, is denoted $T_{\xi}\mathbf{P}^1$. If $\xi$ is a type II point, then $T_{\xi}\mathbf{P}^1$ is isomorphic to the complex projective space $\mathbb{P}^1$.
%

\subsection*{Dynamics on Berkovich space}
In the remainder of this section, $\phi$ denotes an element of $\mathbb{L}(z)$. We are mainly interested in the dynamics of $\phi$, so typically we assume $\deg(\phi) \geq 2$. A rational map $\phi\in\mathbb{L}(z)$ extends uniquely to a map $\phi:\mathbf{P}^1\to\mathbf{P}^1$.  In particular, a degree $d$ holomorphic family $f_t, t \in \mathbb{D}$, with $f_t \in \mathrm{Rat}_d\subset \mathbb{P}^{2d+1}$ for $t \neq 0$, induces a map $\mathbf{f}: \mathbf{P}^1\to\mathbf{P}^1$.

For each type II point $\xi \in \mathbf{P}^1$, there is an induced map $T_{\xi}\phi: T_{\xi}\mathbf{P}^1\to T_{\phi(\xi)}\mathbf{P}^1$; under the isomorphism of the previous paragraph, it is a rational map of one complex projective line to another, and therefore has a degree, which is a nonnegative integer.  A direction $\vec{v}\in T_{\xi}\mathbf{P}^1$ is a \textit{bad direction} of $\phi$ if the image of the corresponding component of $\mathbf{P}\setminus\{\xi\}$ under $\phi$ is the whole space $\mathbf{P}^1$. Otherwise, the image of this component is another such component, and we say $\vec{v}$ is a \textit{good direction} of $\phi$.  The properties of `bad' and `good' are not infinitesimal properties.  In summary, under $\phi$, good directions map to directions, and bad directions map to all of $\mathbf{P}^1$.

A $q$-cycle $\langle\xi\rangle:=\{ \xi_0, \xi_1, \ldots, \xi_{q-1}\} \subset\mathbb{H}$ is \textit{indifferent} if $\deg T_{\xi_0}\phi^q=1$. Otherwise, it is \textit{repelling}. A point $\xi\in\mathbf{P}^1$ belongs to the \textit{(Berkovich) Julia set} $J(\phi)$ if for all neighborhoods $V$ of $\xi$, the set $\cup\phi^n(V)$ omits at most two points. The \textit{(Berkovich) Fatou set} $F(\phi)$ is the complement of $J(\phi)$.

\subsection*{Simple maps.} A map $\phi\in\mathbb{L}(z)$ is \textit{simple} if its Julia set $J(\phi)$ is a singleton; equivalently, up to conjugacy, $\phi$ fixes the Gauss point $\bg$, and the corresponding complex rational map $T_{\bg}\phi$ has degree $\deg(\phi)$; equivalently, $\phi$ has (potentially) good reduction; see \cite[Lemma 10.53]{Baker10}. For a degree $d$ holomorphic family $f_t$, the induced map $\mathbf{f}$ is simple if and only if $M_t\circ f_t\circ M_t^{-1} \to g \in \mathrm{Rat}_d$ for some holomorphic family $M_t $ of degree $1$ rational maps.

\subsection*{Rivera domains}  A distinguishing feature of bicritical maps is the presence of special periodic Fatou components, called Rivera domains, in cases of interest.  A component $U$ of $F(\phi)$ of period $p$ is a \textit{Rivera domain} if $\phi^p: U \to U$ is a bijection; if $p=1$ we say $U$ is a \emph{fixed} Rivera domain. The boundary of such a Rivera domain consists of a finite set of periodic cycles in $J(\phi)$ \cite{Rivera03II}.  For a fixed Rivera domain $U$ which is not an open Berkovich ball, the convex hull $\mathrm{Hull}(\partial U)$ is an invariant finite simplicial tree.  A fixed Rivera domain $U$ is \textit{starlike} if the finite simplicial tree $\mathrm{Hull}(\partial U)$ contains at most one branch point and exactly one fixed point.  \par

Let $U\subset\mathbf{P}^1$ be a connected set. For $\xi\in\mathbf{P}^1\setminus U$, denote by $\vec{v}_{\xi}(U)\in T_{\xi}\mathbf{P}^1$ the direction at $\xi$ containing $U$. To ease notation, if $U=\{\xi'\}$ contains exactly one point, instead of $\vec{v}_{\xi}(U)$, we sometimes write $\vec{v}_{\xi}(\xi')$ the direction at $\xi$ containing $U$. The following result is due to Rivera-Letelier, see \cite{Rivera03II}.  For the definition of multiplicity, see \cite{Baker10}.\par

\begin{theorem}\cite[Theorem 2.4]{Kiwi14}\label{fp-number}
Suppose $\deg(\phi) \geq 2$. Assume that $U$ is a fixed Rivera domain for $\phi$. Let $N(f,U)\ge 0$ be the number of fixed points, counted with multiplicity, in $U\cap\mathbb{P}^1_{\mathbb{L}}$. Then
$$N(f,U)= 2+ \sum_{\xi\in\partial U, \phi(\xi)=\xi}(m_\xi(U)-2),$$
where $m_\xi(U)\ge 0$ is the multiplicity of the direction $\vec{v}_{\xi}(U)\in T_{\xi}\mathbf{P}^1$ as a fixed point of $T_\xi\phi$.
\end{theorem}

\subsection*{The ramification locus} Suppose now $d:=\deg(\phi)\geq 2$.
A point $z \in \mathbb{P}^1_{\mathbb{L}}$ is a \textit{critical point} of $\phi$ if $\phi'(z)=0$ when computed in local coordinates; there are $2d-2$ of them, counted with multiplicity.  The map $\phi$ is \textit{bicritical} if it has exactly two critical points. Just as for maps defined over $\mathbb{C}$, a bicritical map has the form $\phi(z)=M\circ z^d \circ N$ for some $M, N \in \mathrm{Aut}(\mathbb{P}^1_{\mathbb{L}})$.
The \emph{ramification locus} is defined by
$$\mathcal{R}_\phi=\{\xi\in\mathbf{P}^1:\deg T_{\xi}\phi\ge 2\}.$$
If some point in $\mathbf{P}^1$ has local degree achieving the maximum possible value $d$, i.e. if $\phi$ is totally ramified, then the ramification locus $\mathcal{R}_\phi$ coincides with the convex hull of the critical points of $\phi$ \cite[Theorem C]{Faber13I}.

\subsection*{Bicritical rational maps}  We next specialize exclusively to bicritical maps $\phi$ of degree $d \geq 2$. The previously mentioned fact and  \cite[Theorem 9.42]{Baker10} yield the following, which we state for reference and use repeatedly without explicit mention in what follows.

\begin{lemma}\label{ramification}
 Let $\phi:\mathbf{P}^1\to\mathbf{P}^1$ be a bicritical rational map.
\begin{enumerate}
\item The ramification locus $\mathcal{R}_\phi$ is the segment connecting the two critical points.
\item $\deg T_{\xi}\phi=d$ if $\xi\in\mathcal{R}_\phi$ and equals $1$ otherwise.
\item If $\xi \not\in \mathcal{R}_\phi$, a direction is good if and only if it does not meet $\mathcal{R}_\phi$. If $\xi \in \mathcal{R}_\phi$, every direction is good.
\end{enumerate}
\end{lemma}

\subsection*{Dynamical structure of bicritical maps}
Here we generalize Kiwi's structure result of \cite{Kiwi14}. In the next result, case (3) will be of primary interest, since rescaling limits correspond to type II repelling periodic points, see \cite{Kiwi15}.

\begin{proposition}\label{Julia-cycle}
Let $\phi:\mathbf{P}^1\to\mathbf{P}^1$ be a bicritical rational map which is not simple. Then exactly one of the following holds:
\begin{enumerate}
\item There are no cycles in $J(\phi)\cap\mathbb{H}$.
\item There is exactly one cycle $\mathcal{O}$ in $J(\phi)\cap\mathbb{H}$, and $\mathcal{O}$ is indifferent.
\item There is at least one repelling cycle in $J(\phi)\cap\mathbb{H}$.
\end{enumerate}
\end{proposition}
\begin{proof}
Using Lemma \ref{ramification}, the proof of this proposition is exactly the same as the one given in \cite[Proposition 3.1]{Kiwi14}
\end{proof}

\subsection*{A trichotomy}
The next three propositions give the rough structure of the dynamics in each of these cases, respectively.
For completeness, we give the structure results also in cases (1) and (2). They follow from the same arguments in  \cite[Propositions 4.1 and 4.2]{Kiwi14}, and we omit the proofs.
\begin{proposition}\label{attracting}
Let $\phi:\mathbf{P}^1\to\mathbf{P}^1$ be a bicritical rational map which is not simple. Assume that $J(\phi)\cap\mathbb{H}$ contains no periodic points. Then $\phi$ has an attracting fixed point $\xi\in\mathbb{P}^1_{\mathbb{L}}$ and $F(\phi)$ is the immediate basin of $\xi$.
\end{proposition}

\begin{proposition}\label{indifferent}
Let $\phi:\mathbf{P}^1\to\mathbf{P}^1$ be a bicritical rational map which is not simple. Assume that $J(\phi)\cap\mathbb{H}$ contains an indifferent periodic orbit. Then every periodic Fatou component is a fixed Rivera domain.
\end{proposition}
In our main case (3) of interest, there are repelling type II periodic points in $\mathbb{H}$, and we have the following.  A degree $d\ge 2$ polynomial $P(z)\in\mathbb{C}[z]$ is \textit{unicritical} if $P(z)$ is conjugate via an element of $\mathrm{Aut}(\mathbb{C})$ to $z^d+c$.

 \begin{proposition}\label{repelling}
Let $\phi:\mathbf{P}^1\to\mathbf{P}^1$ be a degree $d\ge 2$ bicritical rational map which is not simple. Assume that $J(\phi)\cap\mathbb{H}$ contains a repelling periodic orbit. Then the following statements hold:
\begin{enumerate}
\item $\phi$ has exactly one fixed Fatou component $U$ which is a starlike Rivera domain. The boundary of $U$ is a repelling cycle of type II points $\mathcal{O}=\{\xi_0, \cdots, \xi_{q-1}\}$, $q>1$. The set $\mathrm{Hull}(\partial U)$ contains a unique fixed point $\mathbf{c}$ and the map $T_{\mathbf{c}}\phi$ is a rotation of order $q$.  The periodic Fatou components of higher periods, if they exist, are open Berkovich balls.
\item $\mathbf{P}^1\setminus U$ is the disjoint union of $q$ closed Berkovich balls $B_0,\cdots, B_{q-1}$ labeled such that $\partial B_j=\{\xi_j\}$. One of these balls, say $B_0$, contains both critical points of $\phi$ and the ramification locus. The local degrees are $\deg_{\xi_0}\phi=d$ and $\deg_{\xi_j}\phi=1$ for $j \in \{1, \ldots, q-1\}$. Moreover, for $j=1,\cdots,q-1$, the bad direction at $\xi_j$ is the direction $\vec{v}_{\xi_j}(U)$.
\item For all $\xi_j$, the first-return map $T_{\xi_j}\phi^q$ is a degree $d$ bicritical rational map with a multiple fixed point at the direction $\vec{v}_{\xi_j}(U)$.
\item $\phi$ has at most $2$ repelling cycles in $\mathbb{H}$. Assume that $\phi$ has another type II repelling $q'$-cycle $\mathcal{O}'\subset\mathbb{H}$. Then $q'>q$ and for any $\xi'\in\mathcal{O}'$, the map $T_{\xi'}\phi^{q'}$ conjugate to a degree $d$ unicritical polynomial. Moreover, there exists a critical point of $\phi$ that belongs to a Fatou component which is an open ball $B'$ with  $\partial B'\subset\mathcal{O}'$.
\end{enumerate}
\end{proposition}

\begin{proof}
Statements $(1)-(2)$ follow from the proofs of  conclusion $(4)$ in \cite[Theorem 2]{Kiwi14} and conclusions $(1)-(2)$ in \cite[Lemma 5.1]{Kiwi14}.  Note for $\xi_j\in\mathcal{O}$ and $\xi'\in\mathcal{O}'$ (if it exists) the maps $T_{\xi_j}\phi^q$ and $T_{\xi'}\phi^{q'}$ each have exactly two critical points. Then statements $(3)-(4)$ are consequences of \cite[Theorem 2]{Arfeux16b} and \cite[Proposition 3.4]{Kiwi15}.
\end{proof}

To end this section, we illustrate Proposition \ref{repelling} by the following example.  Below,  we let $\xi_{a,|t^r|}\in\mathbf{P}^1$ be the point corresponding to the  closed ball in $\{z : |z-a|\leq |t^r|\}\subset\mathbb{L}$.  Recall that $\mathbf{P}^1$ is a tree; we use the notation $[x,y]$, $(x,y)$ to denote respectively the closed and open  segments joining $x,y \in \mathbf{P}^1$.
\begin{example}
For $d\ge 2$ consider the bicritical rational map
$$\phi(z)=1+\frac{t}{1-z^d+g(t)}\in\mathbb{L}(z),$$
where $g(t)=\sum_{n\ge 1}a_nt^n\in\mathbb{L}$.  Put $\xi_0:=\mathbf{g}=\xi_{0, |1|}$ and $\xi_1:=\xi_{1, |t|}$.  A computation shows that $\xi_0\mapsto \xi_1 \mapsto \xi_0$ is a repelling $2$-cycle, so here $q=2$ and the boundary of the fixed Rivera domain $U$ is $\partial U=\{\xi_0, \xi_1\}$.
Denote $G(z)=T_{\bg}\phi^2(z)$.  Direct computation of $\phi^2(z)$ and then setting $t=0$ and cancelling shows that
$$G(z)=\frac{(a_1+1)z^d+d-a_1-1}{a_1z^d+d-a_1}$$
and $G(1)=1$ with $G'(1)=1$.  The holes of the degenerate map of degree $d^2$ corresponding to $G$ consist of the indifferent fixed-point $\hat{z}=1$ together with a set $\Lambda$ of $d-1$ other points $h$ satisfying $G(h)=1, G'(h) \neq 0$.

We now show that by judicious choice of the coefficients $a_1, a_2, \ldots$, there exists a repelling 3-cycle $\xi_0' \mapsto \xi_1' \mapsto \xi_2' \mapsto \xi_0'$ with $\xi_0' \in [\bg,0]$.  To do this, we start with the more modest goal of identifying parameters for which the orbit of the origin $x_0:=0$ under $\phi$ satisfies $x_3:=\phi^3(0)=O(t)$; note that this implies $x_3\in \vec{v}_{\bg}(0)$. Computations show
\[ x_1:=\phi(0)=1+t+O(t^2),\;\; x_2:=\phi^2(0)=1+\frac{1}{d+a_1}+O(t).\]
To obtain $x_3\in \vec{v}_{\bg}(0)$ we must arrange so that $x_2$ lies in a bad direction of $T_\bg\phi^2$.  This corresponds to $1+\frac{1}{d+a_1} \in \Lambda$ and we achieve this by picking some $h \in \Lambda$, solving for $a_1$ in terms of $h$, and setting $a_1$ to this value.  To control the next image, we look at the denominator in the defining expression for $\phi$.  We find
\[ 1-(x_2)^d+g=ct+O(t^2)\]
where $c$ depends on $a_2$.   Then
\[ x_3=1+\frac{t}{ct+O(t^2)}=1+\frac{1}{c}+O(t)\]
and by appropriate choice of $a_2$ we find $1+\frac{1}{c}=0$ as required.  We are looking for an appropriate power $0<r\leq 1$ so that $\xi_0'=\xi_{0, |t^r|}$ is periodic of period $3$ and is deployed as shown in the diagram below, where $\mathrm{Hull}(\partial U)$ is indicated with wiggly edges. Since $x_3=O(t)$, we have that $x_3\in\mathbb{L}$ is in the closed ball corresponding to $\zeta=\xi_{0, |t|}\in \mathbf{P}^1$. In summary, we have the following picture:
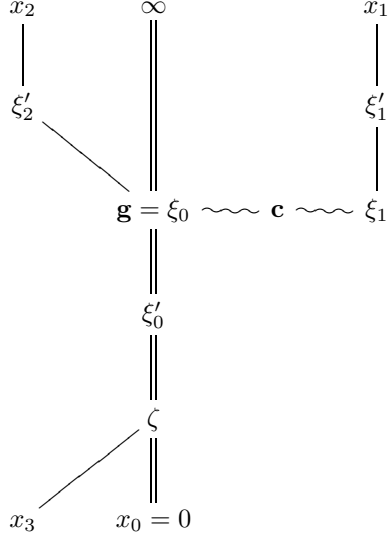
\[
\xymatrix{
x_2\ar@{-}[d]	&\infty\ar@2{-}[dd]	&	&x_1	\ar@{-}[d]	\\
\xi_2'&	&	&	\xi_1'		\ar@{-}[d]	\\
	&\bg=\xi_0\ar@{-}[ul]\ar@{~}[r]	\ar@2{-}[d]&	\bc\ar@{~}[r]&	\xi_1		\\
	&\xi_0'\ar@2{-}[d]	&	&			\\
		&\zeta	\ar@2{-}[d]&	&\\
x_3\ar@{-}[ur]	&x_0	=0&	&			\\	
}
\]
Assuming $r$ is so chosen, the map $\phi$ induces homeomorphisms of segments
\[ [\xi_0, \xi_0'] \to [\xi_1, \xi_1'] \to [\xi_0, \xi_2'] \to [\xi_1, \xi_0'].\]
Let $\rho$ denote the hyperbolic length metric on $\mathbb{H}$, and $\ell$ the length of a path with respect to $\rho$, see \cite{Baker10}. Appropriately normalized we have $\rho(\xi_0, \xi_1)=1$ and $\rho(\xi_{0,1}, \xi_{0, |t^r|})=r$. This metric has the property that $\ell(\phi([a,b])=\ell(a,b)$ if $(a,b)\cap \mathcal{R}_\phi=\emptyset$, and for the bicritical maps considered here, $\ell(\phi([a,b])=d\cdot \ell([a,b])$ if $[a,b] \subset \mathcal{R}_\phi$.  Exploiting this and the fact that
\[ [\xi_1, \xi_0'] = [\xi_1, \xi_0] \cup [\xi_0, \xi_0']\]
we see that the power $r$ must be chosen so that
\[ d\cdot \ell([\xi_0', \xi_0])=\ell([\xi_0', \xi_0])+\ell([\xi_0, \xi_1])\iff dr+1=r\iff r=\frac{1}{d-1}.\]
%
%
%

At $\xi_0'$, the directions corresponding to $0, \infty$ are invariant under the degree $d$ map $T_{\xi'_0}\phi^3$ and have multiplicity each equal to $d$.  It follows that  $T_{\xi_0'}\phi^3$ is conjugate to the unicritical polynomial $z^d$.

When $d=3$, one may take
\[ a_1=\frac{7+3\sqrt{3}i}{3+\sqrt{3}i}, \;\; a_2=-\frac{2(49+25\sqrt{3}i)}{9(3+\sqrt{3}i)}, \;\; a_k=0, k>2.\]

\end{example}

\section{Proof of Theorem \ref{main}}

\subsection{Unboundedness via a degenerating holomorphic family} Here, we show that an ideal boundary point of a type $D$ hyperbolic component $\mathcal{H}$, if it exists, is accessible through a holomorphic family.

A bicritical rational map of type $D$ has two distinct critical points; it also has a non-critical fixed-point. By conjugating so that the two critical points are at zero and infinity and a fixed-point is at $z=1$, It follows that any such map is conjugate to one of finitely many in the two-complex-dimensional algebraic family
\[ \mathcal{F}:=\left\{ \frac{\alpha z^d + \beta}{\gamma z^d + \delta}: \alpha\delta - \beta\gamma = 1, \; \alpha+\beta=\gamma+\delta \right\}\subset \mathrm{Rat}_d.\]
In suitable affine coordinates on the locus $\alpha+\beta=\gamma+\delta$, the family $\mathcal{F}$ becomes a quadric surface in $\mathbb{C}^3$. A generic line intersects $\mathcal{F}$ in two points. Projection from a generic point on $\mathcal{F}$ onto a generic hyperplane in $\mathbb{P}^3$ yields a birational map $\mathcal{F} \to \mathbb{C}^2$. Explicitly, for $u,v \in \mathbb{C}$, set
\[ f_{u,v}(z):={\frac {(  (u+2) v+{u}^{2}+2\,u+2 ) {z}^{d}+u
v+{v}^{2}-1}{ ( -{v}^{2}+uv+ \left( u+1 ) ^{2} \right) {z}^
{d}+ ( u+2\,v+2) v}}=\frac{Az^d+B}{Cz^d+D}.\]
Noting that $AD-BC=\left( v+1 \right)^2  \left( u+v+1 \right)^2$, we conclude
\[ \mathcal{F}\simeq \{f_{u,v} : v +1 \neq 0, u+v+1\neq 0\}\subset \mathbb{C}^2_{u,v}.\]
We compactify this family as $\overline{\mathcal{F}}:=\{[u:v:r] \in \mathbb{P}^2\}$ by adding to $\mathcal{F}$ the \emph{degeneracy locus} $\Delta$ consisting of the line at infinity and the two lines where the resultant vanishes:
\[ \Delta:=\{r=0, v+r=0, u+v+r=0\}.\]

Let $\mathcal{M}_d^\sharp$ be the moduli space of critically marked bicritical rational maps. The forgetful map $\mathcal{M}_d^\sharp \to \mathcal{M}_d$ is proper, so a hyperbolic component in $\mathcal{M}_d$ is bounded if and only if its lift to $\mathcal{M}_d^\sharp$ is bounded. Suppose now $\mathcal{H} \subset \mathcal{M}^\sharp_d$ is a hyperbolic component of type $D$ and $\widetilde{\mathcal{H}} \subset \mathcal{F}$ is its lift to the family $\mathcal{F}$.  The two critical points in $\mathcal{F}$, being located at the origin and at infinity, are therefore marked, hence so are the two corresponding attracting cycles. The multipliers $\lambda, \mu$ of these cycles are complex algebraic functions which are well-defined on $\widetilde{\mathcal{H}}$, yielding an isomorphism
\[\widetilde{\mathcal{H}} \to \mathbb{D}\times\mathbb{D}\]
given by
\[ f_{uv} \mapsto (\lambda(u,v), \mu(u,v)).\]
 We conclude that $\widetilde{\mathcal{H}}$ is a real-algebraic domain in $\mathcal{F}\subset \mathbb{C}_{u,v}$.

We now specialize to the case that $\widetilde{\mathcal{H}}$ is strict type $D$, and we suppose, contrary to the conclusion of Theorem \ref{main}, that the image of $\widetilde{\mathcal{H}}$ in moduli space is unbounded.  Then $\widetilde{\mathcal{H}}$ meets one of the lines in $\Delta$ at some point. In suitable complex affine coordinates $(x,y)$, this point is the origin, and $\widetilde{\mathcal{H}}$ is described by two real-algebraic inequalities.  The Curve Selection Lemma \cite[Lemma 3.1]{Milnor68} implies that there is a real-analytic curve $t \mapsto (x(t), y(t)), 0 \leq t \leq 1$, such that $x(0)=y(0)=0$ and $(x(t),y(t)) \in \widetilde{\mathcal{H}}$ for $0<t\leq 1$. Complexifying this curve as in \cite{Nie-P}, we conclude: \emph{if $\mathcal{H}$ is unbounded, then there exists a holomorphic family $t \mapsto f_{t}:=f_{u(t), v(t)}$ such that for some sequence of parameters $t_k \to 0$, the corresponding maps $f_{t_k}$  belong to $\widetilde{\mathcal{H}}$ and diverge in the moduli space $\mathrm{rat}_d$ as $k \to \infty$}.

\subsection{Bounded multiplier implies existence of a rescaling limit}

We say a holomorphic family $\{f_t\}$ of rational maps is \textit{truly degenerate} if the induced map $\mathbf{f}\in\mathbb{L}(z)$ is not simple. If a sequence $\{f_{t_k}\}$, arising from a holomorphic family $\{f_t\}$ of degree $d\ge 2$ rational maps, diverges in the moduli space $\mathrm{rat}_d$, that is $[f_{t_k}]\to\infty$, then $\{f_t\}$ is truly degenerate.

Here, we show that existence of a bounded multiplier for a truly degenerating family implies the existence of a rescaling limit.

\begin{lemma}\label{repelling-existence}
Let $\{f_t\}$ be a truly degenerate holomorphic family of bicritical rational maps. Assume that there is an $n$-cycle $\langle z(t_k)\rangle$ of $f_{t_k}$ with period $n\ge 2$ and bounded multiplier as $t_k\to 0$. Then the induced map $\mathbf{f}$ has a type II repelling cycle.
\end{lemma}

We remark that Proposition \ref{repelling} shows that the period $q$ of this repelling cycle satisfies $1<q$; in Corollary \ref{Rivera-boundary-period} below we refine this to establish $1<q\leq n$, though this fact is not needed in our proof.

\begin{proof}
The elements of an $n$-cycle of $f_t$ are algebraic functions of the parameter $t$ and define elements of $\mathbb{P}^1_{\mathbb{L}}$ comprising an $n$-cycle $\langle \mathbf{z}\rangle$ of $\mathbf{f}$.

Now suppose the multiplier of some $n$-cycle of $f_{t_k}$ is bounded as $k \to \infty$. Passing to a subsequence, we may assume these cycles are parameterized by a common algebraic function $\langle z(t)\rangle$; we let $\lambda(t)$ denote the multiplier of this cycle. Since $\lambda(t)$ is algebraic and $\lambda(t_k)$ is bounded as $t_k \to 0$, the function $\lambda(t)$ is bounded as $t \to 0$.  It follows that
 $\langle\mathbf{z}\rangle\subset\mathbb{P}^1_{\mathbb{L}}$ is an attracting or indifferent $n$-cycle of $\mathbf{f}$. For each of these two possibilities, we consider the three cases in the conclusion of Proposition \ref{Julia-cycle}, and we will rule out the first two.  Thus the third--the existence of a type II repelling cycle and, hence, of a rescaling limit--must hold.

Suppose $\langle\mathbf{z}\rangle$ is attracting. The cycle $\langle\mathbf{z}\rangle$ is in the Fatou set of $\mathbf{f}$.  If the first case of Proposition \ref{Julia-cycle} holds--that is, there are no periodic cycles in $\mathbb{H}$--then Proposition \ref{attracting} implies that  the entire Fatou set of $\mathbf{f}$ is the immediate basin of an attracting fixed point of $\mathbf{f}$; which is impossible since $n \geq 2$.  If the second case holds, we have an indifferent cycle $\langle \xi \rangle \subset \mathbb{H}$, and Proposition \ref{indifferent} implies that each Fatou component of $\mathbf{f}$ is a fixed Rivera domain. As in our argument for the first case, this contradicts the existence of an attracting $n$-cycle.

Now suppose $\langle\mathbf{z}\rangle$ is indifferent.   The cycle $\langle\mathbf{z}\rangle$ is in the Fatou set of $\mathbf{f}$; see \cite[Proposition 4.24]{Benedetto10} and \cite[Theorem 10.67]{Baker10}.  The first case is ruled out by arguing as in the previous paragraph using  Proposition \ref{attracting}.  We conclude $J(\phi)\cap\mathbb{H}$ contains a cycle.
Next, suppose we are in the second case, so that say $\langle \xi \rangle$ is the unique cycle in $J(\phi)\cap\mathbb{H}$ and is indifferent.
Let $U$ be a Fatou component of $\mathbf{f}$ which contains a point in $\langle\mathbf{z}\rangle$. Then $U$ is periodic, and hence a fixed Rivera domain by Proposition \ref{indifferent}. So $\langle\mathbf{z}\rangle\subset U$.  The boundary $\partial  U$ is also a cycle in $J(\phi) \cap \mathbb{H}$ and so since there is exactly one such cycle we have $\partial U = \langle \xi \rangle$ is indifferent. Hence for each point $\xi\in\partial U$, the degree $\deg T_{\xi}\mathbf{f}^n=1$, and hence the multiplicity of the direction $\vec{v}_{\xi}(U)\in T_{\xi}\mathbf{P}^1$ as a fixed point of $T_\xi\mathbf{f}^n$ is at most $1$. Now consider the number $N(\mathbf{f}^n,U)$ of fixed points of $\mathbf{f}^{n}$ in $U\cap\mathbb{P}^1_\mathbb{L}$. Thus by Theorem \ref{fp-number}, we have $N(\mathbf{f}^n,U)<2$. On the other hand, since each point in $\langle\mathbf{z}\rangle$ is a fixed point of $\mathbf{f}^n$, it follows that $N(\mathbf{f}^n,U)\ge n\ge 2$. This is impossible.
\end{proof}

\subsection{Coarse structure of the Berkovich dynamics}

Lemma \ref{repelling-existence} provides the existence of a rescaling limit. Our next result, Lemma \ref{lemma:coarse-structure}, applies Proposition \ref{repelling} and some basic combinatorial arguments to describe how the Rivera domain $U$ is deployed relative to other dynamical features. Though Lemma \ref{lemma:coarse-structure} describes two cases, our subsequent arguments do not distinguish between them. In the figures below, the ramification locus is indicated with doubled edges, and the edges comprising the simplicial tree $\mathrm{Hull}(\partial U)$ are indicated with wiggly edges.

\begin{lemma}\label{lemma:coarse-structure}
Let $\{f_t\} \subset\overline{\mathcal{F}}$ be a truly degenerate holomorphic family of bicritical rational maps. Assume that there is an $n$-cycle $\langle z(t_k)\rangle$ of $f_{t_k}$ with period $n\ge 2$ and bounded multiplier as $t_k\to 0$. Then the induced map $\mathbf{f}$ has a fixed starlike Rivera domain $U$ with center $\bc$ a type II indifferent  fixed point, and $\partial U = \langle \bg=\xi_0, \xi_1, \ldots, \xi_{q-1}\rangle$ a type II repelling $q$-cycle. Furthermore, either
\begin{enumerate}
\item $\bg \in (\bc, 1)$:
\[
\xymatrix{
& \xi_1\ar@{~}[d] & \infty \ar@2{-}[d]& \\
\xi_2 \ar@{~}[r] & \bc \ar@{~}[r] \ar@{~}[d]&\bg \ar@{-}[r] \ar@2{-}[d]& 1\\
& \xi_{q-1}& 0 & \\
}
\]
 or
\item $\bc \in (\bg, 1)$:
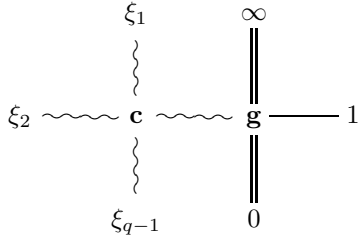
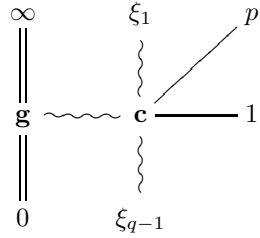
\[
\xymatrix{
\infty  \ar@2{-}[d] &  \xi_1 & p\\
\bg \ar@2{-}[d] \ar@{~}[r] & \bc  \ar@{~}[d] \ar@{-}[r] \ar@{~}[u] \ar@{-}[ur]& 1\\
0 & \xi_{q-1}\\
}
\]

In this case, we have
\begin{enumerate}
\item $(\bg, 1] \subset U$;
\item $\mathbf{f}$ sends the direction $\vec{v}_{\bc}(1)$ onto itself via a homeomorphism;
\item the set $U\cap\mathbb{P}^1_{\mathbb{L}}$ contains exactly $2$ fixed points of $\mathbf{f}$, the point $1$ and a point $p \neq 1$, each of multiplicity $1$;
\item $\bc\in(p,1)$ and $[\bc, p] \cap [\bc, 1] = \{\bc\}$.
\end{enumerate}
\end{enumerate}
\end{lemma}

\begin{proof} Since $f_t$ is assumed truly degenerate, the induced map $\mathbf{f}$ is not simple. The bounded multiplier hypothesis implies, by Lemma \ref{repelling-existence}, the existence of a type II repelling cycle. Proposition \ref{repelling} then gives the existence of a fixed starlike Rivera domain $U$;  we adopt the notation in the statement for the center and boundary of $U$.

We first show $\bg=\xi_0 \in \partial U$. By Proposition \ref{repelling}(3), there is a unique point, denoted by $\bx$, in the intersection of $\partial U$ and the ramification locus $\mathcal{R}_{\mathbf{f}}=[0, \infty]$.  Suppose $\bx \neq \bg$. The center $\bc$ cannot lie on the segment $[\bx,\bg]\cup[\bg,1]$. We have the following configuration in $\mathbf{P}^1$, with the double arrow indicating a subsegment of the ramification locus:
\[
\xymatrix{
 \mathbf{f}(\bx)=\by \ar@{-}[r] & \bc \ar@{-}[r]  & \bx\ar@2{-}[r]&\bg\ar@{-}[r]&1=\mathbf{f}(1).
}
\]

Recall $T_\bx\mathbf{P}^1, T_\by\mathbf{P}^1 \simeq \mathbb{P}^1$.  The map $T_\bx\mathbf{f}: T_\bx\mathbf{P}^1\to T_\by\mathbf{P}^1$ is rational and has degree $d$. By Lemma \ref{ramification}(3), each direction at $\bx$ is a good direction. Since $\mathbf{f}(1)=1$,  the direction $\vec{v}_{\bx}(1)$ maps to the direction $\vec{v}_{\by}(1)$ with local degree $d$ under $T_x\mathbf{f}$. But by Proposition \ref{repelling}(2), $\mathbf{f}$ maps the segment $[\bc,\bx]$ to the segment $[\bc,\by]$. So the direction $\vec{v}_{\bx}(\bc)$ also maps to the direction $\vec{v}_{\by}(\bc)=\vec{v}_{\by}(1)$. This implies the rational map $T_{\bx}\mathbf{f}: T_{\bx}\mathbf{P}^1\to T_{\by}\mathbf{P}^1$ has at least $d+1$ preimages of the point corresponding to the direction $\vec{v}_{\by}(1)$, which is impossible.

Now suppose we are not in Case (1), and that to the contrary we have $\bc \not\in (\bg, 1)$.  Let $\bv$ be the projection of $\bc$ onto $(\bg,1)$ and let $\bw=\mathbf{f}(\bv)$:
\[
\xymatrix{
& \bc\ar@{~}[d]\\
\bg\ar@{~}[r] & \bv \ar@{-}[r]& 1
}
\]
In the starlike Rivera domain $U$, the intervals $(\bc, \xi_i], (\bc,\xi_j]$ are disjoint when $i\neq j$.
The restriction $\mathbf{f}: [\bc,1] \to [\bc, 1]$ is  a homeomorphism fixing each endpoint, since this segment is outside the ramification locus and the endpoints are fixed.  The restriction $\mathbf{f}: (\bc,\xi_0=\bg] \to (\bc, \xi_1]$ is also homeomorphism. In particular it is a homeomorphism from $(\bc,\bv]$ to its image $(\bc, \bw]$.  Thus $\bw \in (\bc, 1)$. The segments $(\bc, \bg=\xi_0]$ and $(\bc, \xi_1]$ must then overlap in a nonempty subsegment of $(\bc,\bv]$; this is impossible.

It remains to verify the claims \emph{(a)-(d)} in Case (2). The direction $\vec{v}_{\mathbf{c}}(1)$ is a good direction. The segment $[\bc, 1]$ has fixed endpoints and does not meet the ramification locus, hence it maps to itself by a homeomorphism, verifying \emph{(b)}. This observation also shows $\emph{(a)}$: otherwise, some $\xi_i \in \partial U \cap (\bc,1)$, $i \neq 0$, whence $\xi_i$ cannot iterate to $\bg = \xi_0$.  By Theorem \ref{fp-number}, the map $\mathbf{f}$ has exactly $2$ fixed points in $U\cap\mathbb{P}^1_{\mathbb{L}}$. Note $T_\bc\mathbf{f}$ is a rotation of order $q\ge 2$. Then $T_\bc\mathbf{f}$ fixes $2$ distinct directions, and \emph{(c),(d)} follow.

\end{proof}

\subsection{Structure of the rescaling limit}  In this subsection, we let $\{f_t\}$ be a holomorphic family such that $f_t\in\mathcal{F}$ if $t\not=0$ and satisfies the assumptions in the Lemma \ref{repelling-existence}. Let $U$ be the fixed Rivera domain given by Lemma \ref{lemma:coarse-structure}.   Then the Gauss point $\mathbf{g}=\xi_0$ lies in the cycle $\partial U=\langle \xi_0, \ldots, \xi_{q-1}\rangle$ of $\mathbf{f}$ and so is fixed under $\mathbf{f}^q$.  Since this cycle is repelling, the reduction of $\mathbf{f}^q$ at $\mathbf{g}$ has degree at least two. It follows that $f_t^q$ converges locally uniformly away from a finite set to a nonlinear rational map $G$; see \cite[Proposition 3.4]{Kiwi15}. In this subsection, we analyze the structure of $G$.
%
%
%

There is a unique isomorphism $T_{\bg}\mathbf{P}^1\to \mathbb{P}^1$ sending the directions at $\bg$ corresponding to $0, \infty, 1 \in \mathbb{P}^1_{\mathbb{L}}$ to the points $0, \infty, 1 \in \mathbb{P}^1$, respectively.  We denote by $G(z)=T_{\bg}\mathbf{f}^q(z)$ the corresponding rational map.  With these normalizations, $G \in \mathcal{F}$.  Lemma \ref{lemma:coarse-structure} and Proposition \ref{repelling}(3) imply the direction $\vec{v}_{\bg}(U)$ corresponds to a multiple fixed point $\hat{z}$ of $G$.  So $\hat{z}\in \mathbb{P}^1, G(\hat{z})=\hat{z}, G'(\hat{z})=1$.
Lemma \ref{lemma:coarse-structure} gives the location of $\hat z$. Indeed, in case (1) of Lemma \ref{lemma:coarse-structure}, $\hat{z}\neq 1$; in case (2),  $\hat{z}=1$.

Expressing $f_t^q$ and $G$ in projective coordinates as a pair of homogeneous degree $d$ polynomials, we have $f_t^q \to H\cdot G$ where $H$ is a homogeneous polynomial whose roots--called the \emph{holes} of the limit--correspond to points where the convergence $f_t^q \to G$ fails to be locally uniform; see \cite[Lemma 4.2]{DeMarco05}. By Proposition \ref{repelling} (2), the bad directions of $\mathbf{f}^q$ at the Gauss point $\bg$ are the directions corresponding to points in $G^{-1}(\hat z)$. By \cite[Lemma 3.17]{Faber13I}, the set of holes of the degenerate rational map corresponding to $G$ is the set $G^{-1}(\hat z)$. Recalling that $\hat{z}$ is fixed, we write the set of holes as $G^{-1}(\hat{z})=\{\hat{z}\}\cup \Lambda$, so that $\Lambda:=G^{-1}(\hat{z})-\{\hat{z}\}$.

Our first result asserts that the holes are disjoint from the critical points of $G$.

\begin{lemma}\label{preimage-simple}
Let $G$ and $\hat z$ be as above. Then
$$G^{-1}(\hat z)\cap\{0, \infty\}=\emptyset.$$
\end{lemma}
\begin{proof}
Note $\hat z\in G^{-1}({\hat z})$ since $\hat z$ is a fixed point of $G$. By Proposition \ref{repelling}, $G'(\hat z)=1$, so $\hat z\not\in\{0,\infty\}$ since $0,\infty$ are critical points.  Suppose to the contrary that $G(0)=\hat{z}$. Then by Proposition \ref{repelling}, the directions $\vec{v}_{\bg}(0)$ and $\vec{v}_{\bg}(\hat z)$ both map to $\vec{v}_{\xi_1}(\bg)$ with degree $d$ and $1$, respectively.  Thus under $G$ we have $0\mapsto \hat{z}$ by degree $d$ and $\hat{z}\mapsto \hat{z}$ by degree $1$. This is impossible since $\deg(G)=d$. The case $G(\infty)=\hat{z}$ is ruled out similarly.
\end{proof}

\subsection{Limits of cycles with bounded multipliers} We continue the setup and notation of the previous subsection.

By assumption, the multiplier $\lambda(t_k)$ of $f_{t_k}$ associated to the cycle $\langle z({t_k})\rangle$ remains bounded as $t_k \to 0$. The proof of Lemma \ref{repelling-existence} shows that $\lambda(t)$ in fact converges, say to $\lambda(0)$, as does the cycle itself, say $\langle z(t)\rangle \to \Gamma \subset \mathbb{P}^1$.  Elementary arguments show easily that $G(\Gamma) \subset \Gamma$. In general, equality need not hold.  The next lemma gives the possibilities in our setting.

\begin{lemma}
\label{limit-infinity}
Let $G, \hat z, \Lambda, \Gamma$ be as above. Then the limit $\Gamma$ of the cycle $\langle z(t)\rangle$ either
\begin{itemize}
\item contains a preperiodic critical point that iterates under $G$ to $\hat{z}$
\item contains a cycle disjoint from $\hat{z}$, or
\item collapses to the singleton $\{\hat{z}\}$.
\end{itemize}
More precisely, $\hat z\in\Gamma$ and exactly one of the following holds.
\begin{enumerate}
\item $\Gamma\cap \Lambda\not=\emptyset$ Then $J(\mathbf{f})$ contains two distinct type II repelling cycles, and the set $\{0,\infty\}\cap\Gamma$ consists of a single point $c'$.  The point $c'$ is a critical point of $G$ and $G^\ell(c')=\hat{z}$ for some $1<\ell< n$.
\item $\Gamma\cap \Lambda=\emptyset$ and $\Gamma\not=\{\hat z\}$. Then $\Gamma\setminus\{\hat z\}$ is a cycle of $G$ and $q\mid n$.
\item $\Gamma=\{\hat z\}$.
\end{enumerate}
\end{lemma}

\begin{proof}

Recall that by Proposition \ref{repelling} (2), the bad directions of $\mathbf{f}^q$ at the Gauss point $\bg$ are the directions corresponding to points in $G^{-1}(\hat z)$.

We begin with a preliminary result. Suppose $\Gamma\cap\Lambda\not=\emptyset$.  We show that under this assumption, $\mathbf{f}$ has a second repelling cycle $\mathcal{O}'$ distinct from that containing $\mathbf{g}$.  Let $w\in\Gamma\cap\Lambda$. The direction $\vec{v}_{\bg}(w)$ is bad for $\mathbf{f}^q$ and so $\bg$ has preimages under $\mathbf{f}^q$ in the direction $\vec{v}_{\bg}(w)$. Hence the Julia set $J(\mathbf{f})$ meets the direction $\vec{v}_{\bg}(w)$ since $\bg\in J(\mathbf{f})$. The cycle $\langle\mathbf{z}\rangle\subset\mathbb{P}^1_{\mathbb{L}}$ is nonrepelling. Hence $\langle\mathbf{z}\rangle$ is in the classical Fatou set of $\mathbf{f}$ (\cite[Proposition 4.24]{Benedetto10}) and hence in the Berkovich Fatou set of $\mathbf{f}$ \cite[Theorem 10.67]{Baker10}. Thus the map $\mathbf{f}$ has periodic Fatou components which are not fixed.  By Proposition \ref{repelling} (1), these periodic components are open balls.
It follows that the corresponding boundary points of these Fatou components form a periodic cycle $\mathcal{O}'\subset J(\mathbf{f})$ for which $\mathbf{g}\not\in\mathcal{O}'$. Therefore,  $J(\mathbf{f})\cap\mathbb{H}$ contains at least two cycles. Then by Proposition \ref{Julia-cycle}, $\mathcal{O'}$ is repelling cycle distinct from that containing $\mathbf{g}$.

Now we claim that $\Gamma\setminus\Lambda\not=\emptyset$. For otherwise, $\Gamma\cap\Lambda=\Gamma\not=\emptyset$, and the hypothesis of the setup of the previous paragraph is satisfied. By the conclusion of the previous paragraph, we have a second repelling cycle $\mathcal{O}'$. Recall that the ramification locus is $\mathcal{R}_{\mathbf{f}}=[0, \infty]$.  We have
$$\mathcal{O}'\subset\bigcup_{w\in G^{-1}(\hat{z})}\vec{v}_{\bg}(w).$$
By Lemma \ref{preimage-simple},  it follows that $\mathcal{O}'$ is disjoint with the ramification locus $\mathcal{R}_{\mathbf{f}}$. Thus $\mathcal{O}'$ is indifferent, which contradicts the fact that $\mathcal{O}'$ is repelling.

The preceding paragraph implies there exists $z\in\Gamma\setminus\Lambda$. Then either $z=\hat z$, and hence $\hat z\in\Gamma$, or the direction $\vec{v}_{\bg}(z)$ is a good direction. In the latter case, since $z\in\Gamma$, there exists $\mathbf{z}_i\in\langle\mathbf{z}\rangle$ such that $\mathbf{z}_i\in\vec{v}_{\bg}(z)$. Note $\mathbf{f}(\bg)$ is contained in the direction $\vec{v}_{\bg}(\hat z)$, see Proposition \ref{repelling}. It follows that $\mathbf{z}_{i+1}\in\vec{v}_{\bg}(\hat z)$ and so $\hat z\in\Gamma$ in this case too.

Finally, we prove  the second part of conclusion (1), and conclusion (2).

Suppose $\Gamma\cap\Lambda\not=\emptyset$. Since the cycle $\mathcal{O}'$ is repelling, it contains an element, say $\xi'$, that lies in the ramification locus. Proposition \ref{repelling} (4) implies there is a critical point $c'\in\{0,\infty\}$ such that the direction $\vec{v}_{\xi'}(c')$ is a good direction of $\mathbf{f}^n$. Note there exists $\mathbf{z}_i\in\langle\mathbf{z}\rangle$ such that $\xi'\in(\mathbf{z}_i, \bg)$.  It follows that $\mathbf{z}_i$ is in the direction $\vec{v}_{\bg}(c')$. Hence $c'\in\Gamma$. Since $\hat z$ is a multiple fixed point, there exists a critical point $c$ of $G$ which is attracted to $\hat z$. Then the direction $\vec{v}_{\bg}(c)$ is contained in the Fatou set $F(\mathbf{f})$ and $\langle\mathbf{z}\rangle\cap\vec{v}_{\bg}(c)=\emptyset$. Hence $c\not\in\Gamma$. Since $G \in \mathcal{F}$,  the set of critical points of $G$ is $\{0,\infty\}$, so $\{0,\infty\}\cap\Gamma=\{c'\}$.

Now we claim that $c'$ iterates under $G$ to $\hat z$. To prove the claim, we show there exists $k\ge 1$ such that $\vec{v}_{\bg}(c')$ maps to $\vec{v}_{\xi_1}(U)$ under the tangent map $T_{\bg}\mathbf{f}^{kq+1}$. First note $\vec{v}_{\bg}(c')$ intersects with the Julia set $J(\mathbf{f})$ since $\vec{v}_{\bg}(c')\cap\mathcal{O}'\not=\emptyset$. Suppose there exists no such $k$. Then by Proposition \ref{repelling} (2), for any $k\ge 0$, the direction $T_{\bg}\mathbf{f}^{kq+1}(\vec{v}_{\bg}(c'))$ is a good direction at $\xi_1$. Again by Proposition \ref{repelling} (1) and (2), it follows that $T_{\bg}\mathbf{f}^{m}(\vec{v}_{\bg}(c'))$ is a good direction at $\mathbf{f}^m(\bg)$ for all $m\ge 0$. Therefore,  $U\bigcap\bigcup_{m\ge 0}\mathbf{f}^m(\vec{v}_{\bg}(c'))=\emptyset$, and hence $\vec{v}_{\bg}(c')$ is contained in $F(\mathbf{f})$, which is a contradiction. Noting that $T_{\bg}\mathbf{f}^{(k+1)q}(\vec{v}_{\bg}(c'))=\vec{v}_{\bg}(U)$, we conclude $G^{k+1}(c')=\hat z$. Now choose the smallest such $k$ and set $\ell=k+1$. Then $G^{\ell}(c')=\hat z$. Note $qk<n$ since $\mathbf{f}^{qk}(\mathbf{z}_i)$ is not in the direction $\vec{v}_{\bg}(c')$. We have $1<\ell=k+1\le qk<n$ since $q\ge 2$.

If $\Gamma\cap\Lambda=\emptyset$ and $\Gamma\not=\{\hat z\}$, we may assume that $\mathbf{z}_0\in\vec{v}_{\bg}(z_0)$ for some $z_0\in\mathbb{C}\setminus\{0,\hat z\}$, and assume $\mathbf{z}_1\in\vec{v}_{\bg}(\hat z)$. For $k \in \mathbb{N}$ let  $z_{kq}\in\Gamma$ be such that $\mathbf{z}_{kq}\in\vec{v}_{\bg}(z_{kq})$.  Note the directions $\vec{v}_{\bg}(z_{kq})$ are away from the bad directions of $\mathbf{f}^q$ at $\bg$. Then $G(z_{kq})=z_{(k+1)q}$, see \cite[Lemma 3.2]{Kiwi15}, and $q\mid n$. It follows that the $z_{kq}$s comprise a cycle of $G$. The conclusion that $\{z_{kq}\}=\Gamma\setminus\{\hat z\}$ then follows by Proposition \ref{repelling}(1).
\end{proof}

Now we do a more elaborate analysis for the case (1) in the above result. Mainly, we focus on the repelling cycle $\mathcal{O}'\not=\partial U$.
\begin{lemma}\label{cycle-period}
Fix the notations as before. Suppose $\Gamma\cap \Lambda\not=\emptyset$ and let $q'$ be the period of $\mathcal{O}'$. Then $q'\mid n$ and $\mathcal{O}'$ consists of the boundaries of the Fatou components containing $\langle\mathbf{z}\rangle$. In particular, let $\xi'\in\mathcal{O}'\cap[0,\infty]$ and let $\mathbf{z}_i\in\langle\mathbf{z}\rangle$ be in the direction $\vec{v}_{\bg}(c')$; then $\mathbf{f}^{kq'}(\mathbf{z}_i)$ is not contained in the direction $\vec{v}_{\bg}(c)$, where $c\in\{0,\infty\}$ and $c\not=c'$.
\end{lemma}
\begin{proof}
By Lemma \ref{limit-infinity} (1), renumbering the points in $\langle\mathbf{z}\rangle$, we assume $\mathbf{z}_0$ and $c'$ are in the same direction at $\bg$. Note $\langle\mathbf{z}\rangle$ is a nonrepelling cycles of $\mathbf{f}$ and hence is contained in the Fatou set $F(\mathbf{f})$. By Proposition \ref{repelling} (1), the Fatou component $\Omega(\mathbf{z}_0)$ containing $\mathbf{z}_0$ is an open ball. Then the boundary $\partial\Omega(\mathbf{z}_0)$ is a periodic point in the Julia set $J(\mathbf{f})$. By Proposition \ref{Julia-cycle} and Lemma \ref{repelling-existence}, $\partial\Omega(\mathbf{z}_0)$ is repelling. Now we claim $\partial\Omega(\mathbf{z}_0)\in \mathcal{O'}$. Since $G$ has a parabolic fixed point $\hat z$, a critical point is attracted to $\hat z$. So this critical point is $c\not=c'$ and there is no periodic point in $(\bg,c)$. Thus by Proposition \ref{repelling} (4),  $\mathcal{O'}\cap(\bg,c')\not=\emptyset$ contains exactly one point $\xi'$. Again by Proposition \ref{repelling} (4), $\mathbf{f}$ has exactly two repelling cycles $\mathcal{O}$ and $\mathcal{O'}$ in $\mathbb{H}$. If $\partial\Omega(\mathbf{z}_0)\notin \mathcal{O'}$, then $\partial\Omega(\mathbf{z}_0)\in \mathcal{O}$. Lemma \ref{limit-infinity} (1) implies that $\Omega(\mathbf{z}_0)$ are contained in the direction $\vec{v}_\bg(c')$. Since $\Omega(\mathbf{z}_0)$ is an open ball, we have $\Omega(\mathbf{z}_0)=\vec{v}_\bg(c')$, which is impossible because $\xi'$ is contained in $\vec{v}_\bg(c')$.

Now by Proposition \ref{repelling}(4), the map $T_{\xi'}\mathbf{f}^{q'}$ is conjugate to a polynomial of degree $d$. Then by Lemma \ref{ramification} (3), for any $m\ge 2$, the bad directions for $\mathbf{f}^m$ at $\xi'$ is $\vec{v}_{\xi'}(c)$. Therefore, $\mathbf{z}_0$ is contained in a good direction of $\mathbf{f}^n$ at $\xi'$. Then the conclusion follows.
\end{proof}

\begin{corollary}\label{Rivera-boundary-period}
Under the assumption of Lemma \ref{repelling-existence}, we have $1<q\le n$.
\end{corollary}
\begin{proof}
Proposition \ref{repelling} implies  $q>1$. To show $q\le n$, we apply Lemmas \ref{limit-infinity} and \ref{cycle-period}. If $\Gamma=\{\hat z\}$, by Proposition \ref{repelling} (1), $\langle\mathbf{z}\rangle\subset U$, and hence  $q\mid n$. If $\Gamma\cap\Lambda\not=\emptyset$, then by Lemma \ref{cycle-period}, it follows that $q<n$ since $q<q'\leq n$ by Proposition \ref{repelling} (4). If $\Gamma\cap\Lambda=\emptyset$ and $\Gamma\not=\{\hat z\}$, it holds by Lemma \ref{limit-infinity} (2).
\end{proof}

Since we will focus on the case that $\langle z(t)\rangle$ is an attracting cycle in next following subsections, we sharpen the conclusion of case (2) in Lemma \ref{limit-infinity} as follows.  We must be careful when examining derivatives of degenerating families.   Written in terms of ratios of homogeneous polynomials, if $f_t \to HG \in \mathbb{P}^{2d+1}$ where $G$ is in lowest terms, then $f_t' \to H^2G'$ in $\mathbb{P}^{4d+1}$; however, $G'$ need not be in lowest terms. The holes of $H^2G'$ are the holes of $HG$ together with a subset of the zeros of the denominator of $G$, i.e. the poles of $G$.

\begin{lemma}\label{nonrepelling-cycle}
Fix the notations as before. Assume $\langle z(t)\rangle$ is an attracting cycle. If $\Gamma\cap \Lambda=\emptyset$ and $\Gamma\not=\{\hat z\}$. Then $\Gamma\setminus\{\hat z\}$ is a nonrepelling cycle of $G$.
\end{lemma}
\begin{proof}
By Lemma \ref{limit-infinity} (2), $\Gamma\setminus\{\hat z\}$ is a cycle of $G$. We show this cycle is nonrepelling. If $\infty\in\Gamma\setminus\{\hat z\}$, then $\Gamma\setminus\{\hat z\}$ is superattracting since $\infty$ is a critical point of $G$. Now we deal with the case that $\infty\not\in\Gamma\setminus\{\hat z\}$. Since $f_t^q$ converges to $G$ locally uniformly on $\mathbb{P}^1$ off finitely many points, there exists a homogeneous polynomial $H$ such that $f_t^q\to HG$ in $\mathbb{P}^{2d^q+1}$. It is easy to check that $(f_t^q)'\to H^2G'$ in $\mathbb{P}^{4d^{q}+1}$. Thus the holes of $H^2G'$ are contained in $\Lambda\cup\{\hat z\}$ and the poles of $G$.
%
%
%
For any point $x\in\Gamma\setminus\{\hat z\}$, let $x(t)\in\langle z(t)\rangle$ such that $x(t)$ converges to $x$ as $t\to 0$. Since $\Gamma\cap\Lambda=\emptyset$ and $\infty\not\in\Gamma\setminus\{\hat z\}$, then $x$ is not a hole of $H^2G'$. Now regard $(f_t^q)'$ and $G'$ as rational maps in lowest terms. Then $(f_t^q)'(x(t))\to G'(x)$ as $t\to 0$, see \cite[Lemma 2.6]{DeMarco07}. Now renumber the points in $\langle z(t)\rangle$ such that $z_0(t)\to z_0\in\Gamma\setminus\{\hat z\}$. Then by Lemma \ref{limit-infinity} and the chain rule,
$$(f_t^{n})'(z_0(t))=(f_t^q)'(z_{n/q-1}(t))\cdots(f_t^q)'(z_{q}(t))(f_t^q)'(z_{0}(t)).$$
It follows that $(f_t^{n})'(z_0(t))\to (G^{n/q})'(z_0)$. Since $|(f_t^{n})'(z_0(t))|<1$, we have that $|(G^{n/q})'(z_0)|\le 1$. Note $G^{n/q}(z_0)=z_0$ since $f_t^n(z_0(t))=z_0(t)$. Thus the cycle $\Gamma\setminus\{\hat z\}$ is nonrepelling.
\end{proof}

\subsection{Limits of a pair of attracting cycles}

We continue the setup and notation of the previous two subsections.  However, from now on, we assume in addition that the family $f_t$ has a degenerate sequence $f_{t_k}$ possessing two distinct \emph{attracting} cycles of periods at least $2$.

Lemma \ref{limit-infinity} gives three possibilities for the limit of each these cycles. With two indistinguishable cycles, we get a priori six possibilities to analyze in total. The following result constrains the limiting map $G(z)$ in certain cases, and serves to rule out some possibilities.

\begin{lemma}\label{infinity-infinity}
Let $\{f_t\}$ and $\{f_{t_k}\}$ be as above. Assume that $\langle z(t_k)\rangle$ and $\langle w(t_k)\rangle$ are two distinct attracting cycles of periods at least $2$. Suppose $\langle z(t_k)\rangle\to\{\hat z\}$ as $k \to \infty$. Then $\langle w(t_k)\rangle\to\{\hat z\}$ also. Moreover, both $\langle z(t_k)\rangle$ and $\langle w(t_k)\rangle$ have period $q$, and the point $\hat{z}$ is a parabolic fixed point $G(z)$ with multiplicity $3$.
\end{lemma}
\begin{proof}
For the unique fixed Rivera domain $U$ of $\mathbf{f}$, Theorem \ref{fp-number} implies that $\mathbf{f}$ has $2$ fixed points in $U\cap\mathbb{P}^1_{\mathbb{L}}$. Let $n$ be the period of the cycle $\langle z(t_k)\rangle$. Now consider the number $N(\mathbf{f}^{n},U)$ of the fixed points of $\mathbf{f}^{n}$ in $U\cap\mathbb{P}^1_{\mathbb{L}}$. By Theorem \ref{fp-number}, we have on the one hand
$$N(\mathbf{f}^{n},U)=2+q(m-2),$$
where $m$ is the multiplicity of $\hat z\in\mathbb{C}\setminus\{0\}$ as a fixed point of $G(z)$. Since  $\langle z(t_k)\rangle\to\{\hat z\}$, it follows that $\langle\mathbf{z}\rangle\subset U$. For otherwise, Proposition \ref{repelling}(2) implies there are points distinct from $\hat z$ in the limit of $\langle z(t_k)\rangle$. Note each point in $\langle\mathbf{z}\rangle$ is a fixed point of $\mathbf{f^n}$. So we have on the other hand that $N(\mathbf{f}^{n},U)\ge 2+n$. Hence $m\ge 3$ and so the parabolic fixed point $\hat z$ attracts in its basin at least two critical points of $G$.
Since $G$ has exactly two critical points, it follows that $m=3$, equivalently $n=q$.

We now consider the three possibilities for the limit of the cycle $\langle w(t_k) \rangle$ given by Lemma \ref{limit-infinity}.  The previous paragraph shows that each critical point of $G$ has infinite forward orbit and converges to the multiplicity $3$ parabolic fixed-point $\hat{z}$. If the limit contains a cycle, it must be non-repelling, by the proof of Lemma \ref{nonrepelling-cycle}.  By \cite{Fatou, Mane93}, to a non-repelling cycle is associated yet another critical point of $G$, and so this is impossible; to see this, consider each of the attracting, parabolic, Siegel, and Cremer cases. The limit cannot contain a preperiodic critical point of $G$ either. Hence this limit must collapse to $\hat{z}$.

We conclude that both cycles collapse and hence collide to $\hat{z}$.
\end{proof}


\subsection{Fatou-Shishikura Inequality}
To finally prove our main result, we will apply the Fatou-Shishikura inequality (FSI) and its  version refined by A. Epstein (refined FSI) to obtain impossible on constraints the number of nonrepelling cycles of the limiting map $G$.

\begin{theorem}[FSI]\cite[Corollary 1]{Shishikura87}\label{FSI}
Let $f \in \mathrm{Rat}_d$, $d \geq 2$. Then $f$ has at most $2d-2$ nonrepelling cycles.
\end{theorem}

\begin{theorem}[Refined FSI]\cite[Theorem 1]{Epstein99}\label{refined-FSI}
Let $f \in \mathrm{Rat}_d$, $d \geq 2$.  For a cycle $\langle z\rangle\subset\mathbb{P}^1$, define
$$\gamma_{\langle z\rangle}:=
\begin{cases}
0&\text{if}\  \langle z\rangle\ \text{is\  repelling\ or\ superattracting},\\
1&\text{if}\  \langle z\rangle\ \text{is\  attracting\ or\ irrationally indifferent},\\
\nu &\text{if}\  \langle z\rangle\ \text{is\  parabolic-repelling},\\
\nu+1& \text{if}\  \langle z\rangle\ \text{is\  parabolic-attracting\ or\ parabolic-indifferent},
\end{cases}$$
where $\nu$ is the corresponding degeneracy if $\langle z\rangle$ is parabolic. Set
$$\gamma(f)=\sum_{\langle z\rangle\subset\mathbb{P}^1}\gamma_{\langle z\rangle},$$
and let $\delta(f)$ be the number of infinite tails of critical orbits.
Then
$$\gamma(f)\le\delta(f).$$
\end{theorem}

\subsection{Hyperbolic components of strict type D are bounded}
Recall that $\mathcal{M}_d$ is the moduli space of  bicritical rational maps  of degree $d\ge 2$. 
Theorem \ref{main} is a consequence of the following result.

\begin{theorem}\label{bounded-family}
Let $\mathcal{H}\subset\mathcal{M}_d$ be a hyperbolic component possessing two distinct attracting cycles of periods at least $2$. Let $\{f_t\}_{t \in \mathbb{D}}\subset\overline{\mathcal{F}}$ be a holomorphic family
such that $[f_{t_k}]\in\mathcal{H}$ for some sequence $t_k\to 0$.
Then $[f_t]$ is bounded in $\mathcal{M}_d$.
\end{theorem}

\begin{proof}
Suppose that $[f_t]\to\infty$ in $\mathcal{M}_d$ as $t\to\infty$. Then the induced map $\mathbf{f}$ is not simple. Let $\langle z(t_k)\rangle$ and $\langle w(t_k)\rangle$ be the two distinct attracting cycles of periods at least $2$ of $f_{t_k}$. Assume that $\langle z(t)\rangle\to\Gamma^1$ and $\langle w(t)\rangle\to\Gamma^2$ as $t \to 0$. Then by Lemma \ref{limit-infinity}, $\hat z\in\Gamma^1\cap\Gamma^2$.

Recall that $U$ denotes the unique fixed Rivera domain of $\mathbf{f}$, and $q\ge 2$  the period of a point in $\partial U$.  Lemma \ref{limit-infinity} gives three possibilities for $\Gamma^1$, which we treat in turn.

\textbf{Case 1: $\Gamma^1=\{\hat z\}$.} Then by Lemma \ref{infinity-infinity}, we have $\Gamma^2=\{\hat z\}$. Moreover, both $\langle z(t)\rangle$ and $\langle w(t)\rangle$ have period $q$ and the cycles $\langle\mathbf{z}\rangle$ and $\langle\mathbf{w}\rangle$ are in the Rivera domain $U$. It follows that
$$N(\mathbf{f}^q,U)\ge 2+2q.$$
On the other hand, in this case again by Lemma \ref{infinity-infinity} the point $\hat z\in\mathbb{C}\setminus\{0\}$ is a parabolic fixed point $G(z)$ with multiplicity $3$. Hence, by Theorem \ref{fp-number},
$$N(\mathbf{f}^q,U)=2+q.$$
It is impossible since $q\ge 2$.

\textbf{Case 2: $\Gamma^1\cap\Lambda\not=\emptyset$.} By Proposition \ref{repelling}(4) and Lemma \ref{limit-infinity}, the map $\mathbf{f}$ has a type II repelling cycle $\mathcal{O}'$ of period $q'>q$ and there exists $c'\in\{0,\infty\}$ such that $c'\in\Gamma^1$ and is preperiodic under $G$. Since $G$ also has a parabolic fixed point at $\hat{z}\in\mathbb{C}\setminus\{0\}$, Lemmas \ref{limit-infinity},  \ref{infinity-infinity} and the refined FSI imply that $\Gamma^2\cap\Lambda\not=\emptyset$ and $c'\in\Gamma^2$ too.

By Proposition \ref{repelling}(4), at the point $\xi'\in\mathcal{O'}\cap[0,\infty]$, the map $T_{\xi'}\mathbf{f}^{q'}$ is conjugate to a degree $d$ unicritical polynomial.
Now choose a suitable holomorphic family $M_t$ of degree $1$ rational map such that for the induced map $\mathbf{M}\in\mathrm{Aut}(\mathbb{P}^1_{\mathbb{L}})$, $\mathbf{M}(\xi')=\bg$ and the map $T_{\bg}(\mathbf{M}\circ\mathbf{f}^{q'}\circ\mathbf{M}^{-1})$ is a unicritical polynomial. Recall that there exist exactly two repelling cycles in $J(\mathbf{f})$.  Since $\xi'\in\mathcal{O'}$, by Lemma \ref{cycle-period}, renumbering the points in $\langle\mathbf{z}\rangle$ and $\langle\mathbf{w}\rangle$, we obtain $\mathbf{z}_0\in\langle \mathbf{z}\rangle$ and $\mathbf{w}_0\in\langle \mathbf{w}\rangle$ such that the boundaries of the corresponding periodic Berkovich Fatou components $\Omega(\mathbf{z}_0)$ and $\Omega(\mathbf{w}_0)$ are the point $\xi'$. Note $\{\mathbf{M}(\mathbf{z}_{kq'})\}_{k\ge 0}$ and $\{\mathbf{M}(\mathbf{w}_{kq'})\}_{k\ge 0}$ are two nonrepelling cycles of $\mathbf{M}\circ\mathbf{f}^{q'}\circ\mathbf{M}^{-1}$ in $\mathbb{L}$ and hence in $\mathbf{P}^1$. Lemma \ref{cycle-period} implies that $\{M_t(z_{kq'}(t))\}$ and $\{M_t(w_{kq'}(t))\}$ converge to two (not necessarily distinct) cycles in $\mathbb{C}$
of the corresponding polynomial. Moreover, noting the holes of the limit of $M_t\circ f_t^{q'}\circ M_t^{-1}$ are at $\infty$ and applying the proof of Lemma \ref{nonrepelling-cycle}, we have that these cycles are nonrepelling. The FSI implies that a unicritical polynomial has at most one nonrepelling cycle in $\mathbb{C}$. Hence $\{M_t(z_{kq'}(t))\}$ and $\{M_t(w_{kq'}(t))\}$ converge to the same cycle. We next appeal to two elementary lemmas from complex analysis regarding limits of cycles under a locally uniformly convergent sequence of maps. By \cite[Lemma 2.5]{Nie-P} and \cite[Lemma 1]{Epstein00}, we know this cycle is parabolic-attracting or parabolic-indifferent. The refined FSI implies that a unicritical polynomial cannot have such a cycle.

\textbf{Case 3: $\Gamma^1\cap\Lambda=\emptyset$ and $\Gamma^1\not=\{\hat z\}$.} By Lemma \ref{nonrepelling-cycle}, $\Gamma^1\setminus\{\hat z\}$ is a nonrepelling cycle.  By symmetry, it only remains to consider the case that $\Gamma^2\setminus\{\hat z\}$ is also a nonrepelling cycle. Note $G(z)$ has a parabolic fixed point. The FSI implies these two nonrepelling cycles collide, and hence they are same, which is again a parabolic-attracting or parabolic-indifferent cycle of $G$. Again, it contradicts to the refined FSI.
\end{proof}

\bibliographystyle{siam}
\bibliography{references}
\end{document}